\documentclass[a4paper,11pt,twoside,reqno]{amsart}

\usepackage[utf8]{inputenc}
\usepackage[plainpages=false,pdfpagelabels=true]{hyperref}
\usepackage{amssymb,amsthm}
\usepackage[margin=1in]{geometry}
\usepackage{slashed}
\usepackage{comment}
\usepackage{caption}
\newtheorem{Satz}{Theorem}[section]
\newtheorem{Prop}[Satz]{Proposition}
\newtheorem{Lem}[Satz]{Lemma}

\newtheorem{Thm}[Satz]{Theorem}

\theoremstyle{definition}
\newtheorem{Dfn}[Satz]{Definition}
\newtheorem{Bem}[Satz]{Remark}
\newtheorem{Bsp}[Satz]{Example}
\newcommand{\tr}{\operatorname{Tr}}

\newcommand{\vol}{{\operatorname{vol}}}
\newcommand{\dv}{\text{ }dv}

\newcommand{\norm}[1]{\left \lVert #1 \right \rVert}
\allowdisplaybreaks[1]

\renewcommand{\epsilon}{\varepsilon}

\newcommand{\R}{\ensuremath{\mathbb{R}}}
\newcommand{\N}{\ensuremath{\mathbb{N}}}

\newcommand{\s}{\ensuremath{\mathbb{S}}}
\numberwithin{equation}{section}

\usepackage{color}

\title{A framework for producing harmonic maps}
\author{Volker Branding}
\date{\today}
\address{University of Rostock, Institute of Mathematics\\
Ulmenstraße 69, 18057 Rostock, Germany}
\email{volker.branding@uni-rostock.de}

\author{Anna Siffert}
\address{Universität M\"unster, Mathematisches Institut\\
Einsteinstr. 62\\
48149 M\" unster\\
Germany}
\email{asiffert@uni-muenster.de}


\begin{document}
\begin{abstract}
We present a general framework for producing families of harmonic maps from the unit ball into the Euclidean sphere starting from a fixed harmonic map. In particular, our analysis provides a theoretical background for a generalized radial projection which was recently introduced by Nakauchi.

Our approach is based on work of Toth who established a machinery that generates eigenmaps between spheres from a given eigenmap.

Finally, we point out how these families of harmonic maps can be used to construct solutions to associated variational problems such as
intrinsic biharmonic and triharmonic maps, $p$-harmonic maps as well as extrinsic polyharmonic maps. 
\end{abstract} 

\maketitle

\section{Introduction}
Geometric variational problems are a modern direction of research within differential geometry. Primary objects of study are various energy functionals whose critical points characterize particular maps between Riemannian manifolds. Such energy functionals may contain derivatives of the map up to an arbitrary order, however, the larger the number of derivatives, the more technical details need to be sorted out.

Regarding the energy functionals that are natural from a geometric point of view we mention the energy of a map
\begin{align}
\label{harmonic}
E(\phi)=\frac{1}{2}\int_M|d\phi|^2\dv,    
\end{align}
where \(\phi\colon M\to N\) is a smooth map from a compact Riemannian manifold \((M,g)\) to a Riemannian manifold \((N,h)\). 
The critical points of \eqref{harmonic} are determined by the vanishing of the so-called \emph{tension field} which is given by

\begin{align*}
    0=\tau(\phi):=\tr\bar\nabla d\phi,\qquad \tau(\phi)\in\Gamma(\phi^\ast TN).
\end{align*}
Here, \(\bar\nabla\) represents the connection on the pull-back bundle \(\phi^\ast TN\). Solutions of \(\tau(\phi)=0\) are called \emph{harmonic maps}.
The harmonic map equation \(\tau(\phi)=0\) represents a second order semilinear partial differential equation.
For additional information on harmonic maps, we refer to the book \cite{MR2044031}.

In the case of maps \(u\colon M\to\s^n\subset\R^{n+1}\) to the Euclidean sphere, the 
vanishing of the tension field is equivalent to the identity
\begin{align*}
    \Delta u+|\nabla u|^2u=0.
\end{align*}

\smallskip

One specific critical point of the energy \eqref{harmonic} which has been extensively studied in the literature is the
equator map, i.e. the map given by
\begin{align}
\label{equator}
u^{(1)}_{\star}\colon B^m&\rightarrow\s^m\subset\mathbb{R}^m\times\mathbb{R},\\ 
  \notag x&\mapsto\big(u^{(1)}(x),0\big),
\end{align}
where 
$$u^{(1)}:B^m\rightarrow\s^{m-1}, \qquad x\mapsto\frac{x}{\norm x}$$ is the well-studied radial projection map.
Clearly, $u^{(1)}\in W^{1,2}(B^m,\R^{m+1})$ for \(m>2\). 
The equator map $u^{(1)}_{\star}$ is a well-known example of a harmonic map, see for example \cite{MR705882}.

Recently, Nakauchi \cite{MR4593065} constructed a generalized radial projection map $u^{(\ell)}:B^m\rightarrow\s^{m^\ell-1}$, where $\ell\in\mathbb{N}$ is a parameter and $u^{(1)}$ coincides with the radial projection map \(u^{(1)}\) introduced above. More details on this generalized radial projection map can be found in Subsection \ref{sub-Nakauchi}.
The generalized radial projection map gives rise to a generalized equator map as follows
\begin{align}
\label{equator-generalized}
u^{(\ell)}_{\star}\colon B^m&\rightarrow\s^{m^{\ell}}\subset\mathbb{R}^{m^{\ell}}\times\mathbb{R},\\ 
  \notag x&\mapsto\big(u^{(\ell)}(x),0\big).
\end{align}

We would like to point out that one motivation to study maps 
such as \(u^{(\ell)}\), in particular for \(\ell=1,2\), in the 1980's by Giaquinta and Ne\v{c}as
was to find counterexamples to the regularity of weak solutions of elliptic systems that were investigated at that time, see for example \cite{MR581329,MR1401417,MR566246}.

\smallskip

One of the advantages of the generalized equator map, as opposed to the equator map, is the fact that it contains an additional parameter \(\ell\in\N\), which gives an additional degree of freedom, in particular, when it comes to deformations of the generalized equator map.

\medskip

In this paper, we present a technique for constructing additional \lq generalized equator maps\rq\, derived from the so-called eigenmaps, which are sphere-valued harmonic maps that possess constant density
\(e(\phi):=\frac{1}{2}|d\phi|^2\).
Our technique is inspired from the work of Toth \cite{MR1154839,MR1241953}, where a general theory for producing eigenmaps from a given eigenmap is established.
In order to state our main theorem we briefly need to introduce the following mathematical objects. 
First, recall that a map $\varphi:\s^m\rightarrow\s^r$ between Euclidean spheres is an eigenmap of degree $p$, or short $\lambda_p$-eigenmap, if and only if all components of $\varphi$ are spherical harmonics of order $p$ on $\s^m$, see e.g. \cite{MR1242555}. 
Further, recall that a spherical harmonic on $\s^m$ of order $p$ is the restriction to $\s^m$ of a homogeneous harmonic polynomial of degree $p$ in $m+1$ variables. 
By $\mathcal{H}_q^{\R^n}$ we represent the space of homogeneous harmonic polynomials of order $q$ in $n$ variables and $(\mathcal{H}_q^{\R^n})^{\star}$ is the corresponding dual space.
Below $D$ denotes a nonzero module homomorphism of an orthogonal $\mbox{SO}(m+1)$-module $W$ into $(\mathcal{H}_p^{\R^n})^{\star}\otimes\mathcal{H}_q^{\R^n}$.  
We write \((~~)^D\) in order to denote the application of the module homomorphism \(D\) to the expression written in brackets.

\smallskip

The following is our main theorem:

\begin{Thm}
\label{thm:main}
Let $\varphi:\s^{m}\rightarrow\s^d$ be an eigenmap and denote the associated $(d+1)$-tuple consisting of the corresponding homogeneous harmonic polynomials
by $\Phi:\mathbb{R}^{m+1}\rightarrow\mathbb{R}^{d+1}$.
Let $p$ be the degree of the homogeneous polynomials which are the components of $\Phi$.
Set 
\begin{align*}
 \nu^{(p)}(x):=\frac{\Phi(x)}{\norm{x}^{p}}.
\end{align*}
Let $D$ be a nonzero module homomorphism of an orthogonal $\mbox{SO}(m+1)$-module $W$ into $(\mathcal{H}_p^{\mathbb{R}^{m+1}})^{\star}\otimes\mathcal{H}_q^{\mathbb{R}^{m+1}}$.
We define $\nu^{(q)}$ as follows
\begin{align*}
\nu^{(q)}\colon&\mathbb{\R}^{m+1}\setminus\{0\}\rightarrow\mathbb{R}^{(d+1)\dim W},\\
&x \mapsto\frac{1}{\norm{x}^{q}}\big(\norm{x}^{p}\nu^{(p)}(x)\big)^D.
\end{align*}    
The map $\nu^{(q)}$ is a homogeneous polynomial in the variables $(y_i)_{i=1,\dots,m}$, where $y_i\colon=\frac{x_i}{\norm{x}}$
and satisfies the following properties:
\begin{enumerate}
    \item $\nu^{(q)}$ is sphere-valued, i.e. $\norm{\nu^{(q)}(x)}^2=1$ for all $x\in\mathbb{\R}^{m+1}\setminus\{0\}$;
    \item the energy density of $\nu^{(q)}$ is given by
    \begin{align*}
 \norm{\nabla \nu^{(q)}(x)}^2=\frac{q(q+m-2)}{\norm{x}^2};
\end{align*}
 \item $\nu^{(q)}$ is a harmonic map, i.e. it is a solution of the differential equation
 \begin{align*}
     \Delta \nu^{(q)}+\norm{\nabla  \nu^{(q)}}^2\nu^{(q)}=0.
 \end{align*}
\end{enumerate}
\end{Thm}

Clearly, each $\nu^{(q)}:B^{m+1}\rightarrow\s^{(d+1)\dim W -1}$ constructed in Theorem\,\ref{thm:main} gives rise to a generalized equator map
via
\begin{align}
\label{equator-generalized_2}
\nu^{(q)}_{\star}\colon B^{m+1}&\rightarrow\s^{(d+1)\dim W}\subset\mathbb{R}^{(d+1)\dim W}\times\mathbb{R},\\ 
  \notag x&\mapsto\big(\nu^{(q)}(x),0\big).
\end{align}

\smallskip

\subsection{Applications} 
In order to highlight the importance of Theorem \ref{thm:main}
we present a number of mathematical results established in a series of articles \cite{BS25tri,BS25kstable,BS25bistable} in which 
generalized radial projections were used to construct biharmonic, triharmonic and extrinsic k-harmonic maps. In addition, the stability of these maps was also studied and the first example of a stable biharmonic to the sphere
was found using this approach \cite{BS25bistable}. In order to shorten the notation we will always denote the dimension of the target by \(n\).

Let us first recall the most notable generalization of the energy functional, which is
the intrinsic bienergy functional, or simply the bienergy. It is defined by
\begin{align}
\label{eq:bienergy}
E_2(\phi):=\frac{1}{2}\int_M |\tau (\phi)|^2 \dv.    
\end{align}

Critical points of the bienergy functional $E_2$ are called \textit{intrinsic biharmonic maps}, or simply \textit{biharmonic maps}, and
are characterized by the vanishing of the  \textit{bitension field} $\tau_2$, which is given by
\begin{align}
\label{tau2}    
\tau_2(\phi):= \bar\Delta \tau (\phi) +\textrm{Tr}_g R^N(\tau(\phi), d\phi) d\phi.
\end{align}
Here, $\bar\Delta$ denotes the rough Lapacian acting on sections of the pullback bundle $\phi^\ast TN$.
The biharmonic map equation \(\tau_2(\phi)=0\) represents a fourth order semilinear elliptic partial differential equation and the large number of derivatives introduces substantial mathematical difficulties.

It is evident that harmonic maps serve as examples of biharmonic maps.
From now on, we will limit our focus to non-harmonic biharmonic maps, referred to as \textit{proper} biharmonic maps.
For more details on biharmonic maps in Riemannian geometry we refer to the recent book \cite{MR4265170}.

\smallskip

In the particular case of maps to the sphere, the bienergy \eqref{eq:bienergy} takes the form
\begin{align*}
  E_2(u)=\frac{1}{2}\int_M\big(|\Delta u|^2-|\nabla u|^4)\dv, 
\end{align*}
where \(u\colon M\to\s^n\subset\R^{n+1}\), and whose critical points are given by
\begin{align}
\label{eq:biharmonic-intro}
\Delta^2u+2\operatorname{div}\big(|\nabla u|^2\nabla u\big)
-\big(\langle\Delta^2u,u\rangle-2|\nabla u|^4\big)u=0.
\end{align}

\smallskip

Despite being natural from a geometric perspective, the bienergy \eqref{eq:bienergy} lacks coercivity, making it difficult to show the existence of critical points.
However, the generalized radial projection map is a powerful tool to construct biharmonic maps to the sphere. This approach was initiated in \cite{MR4830603,MR4076824}, the most general result in this regard was obtained by the authors in \cite[Theorem 3.2]{BS25tri},
see also \cite{MR5037916}, and reads as:

\begin{Satz}
\label{thm:biharmonic-main}
Let $m,q\in\mathbb{N}$.
The map
\(Q:\mathbb{R}^{m}\setminus\{0\}\rightarrow \s^{n}\)
given by
\begin{align*}
Q:=\big(\sin\alpha\cdot \nu^{(q)},\cos\alpha\big), \qquad \alpha\in (0,\frac{\pi}{2})
\end{align*}
where \(\nu^{(q)}:\mathbb{R}^{m}\setminus\{0\}\rightarrow \R^{n}\) is defined in Theorem\,\ref{thm:main},
is a proper biharmonic map if and only if
the following equation is satisfied 
\begin{align}
\label{eq:constraint-biharmonic}
\sin^2\alpha=\frac{q(q+m-2)+2m-8}{2q(q+m-2)}.
\end{align}    
\end{Satz}

Note that for \(m\geq 3\) there exists \(q\in\N\) with \(q\leq m\) such that the condition \eqref{eq:constraint-biharmonic} can be satisfied, see \cite[Corollary 3.6]{BS25tri}

\medskip

It is then natural to ask if the biharmonic maps provided by Theorem \ref{thm:biharmonic-main} are stable or unstable critical points of the bienergy \eqref{eq:bienergy}. This question was answered by the authors in \cite{BS25bistable}, where the following criterion for stability was established:

\begin{Satz}
\label{thm:intrinsic}
The proper 
biharmonic map \(Q\colon B^m\to\s^{n},m\geq 5\), 
given by 
\begin{align*}
Q:=\big(\sin\alpha\cdot \nu^{(q)},\cos\alpha\big), \qquad \alpha\in (0,\frac{\pi}{2})
\end{align*}
is 
strictly stable if 
the inequality
\begin{align*}
m> 2(\sqrt{12q^2+30q+12}+3q+6)
\end{align*}
is satisfied.
\end{Satz}

Another higher order variational problem for maps between Riemannian manifolds is given by the theory of triharmonic maps. Here, the starting point is the trienergy which is defined as follows
\begin{align}
\label{eq:trienergy}
    E_3(\phi)=\int_M|\bar\nabla\tau(\phi)|^2\dv.
\end{align}
The critical points of \eqref{eq:trienergy} are characterized
by the vanishing of the tritension field
\begin{align*}
    0=\tau_3(\phi):=\bar\Delta^2\tau(\phi)-
    \sum_{j=1}^m R^N(\bar\nabla_{e_j}\tau(\phi),\tau(\phi))d\phi(e_j)
-\sum_{j=1}^m R^N(\bar\Delta\tau(\phi),d\phi(e_j))d\phi(e_j),
\end{align*}
where \(\{e_j\}_{1\leq j\leq m}\) represents a local orthonormal basis of \(TM\).
For maps \(u\colon M\to\s^n\subset\R^{n+1}\) the tritension field can be simplified as well, see \cite[Prop. 4.2]{BS25tri} for the precise details.
Using Theorem\,\ref{thm:main} we are also able to construct  triharmonic maps to spheres based on \cite[Theorem 4.5]{BS25tri}, i.e.

\begin{Satz}
\label{thm:triharmonic-main}
Let $m,q\in\mathbb{N}$ and \(\nu^{(q)}\) be given as in Theorem\,\ref{thm:main}.
The map
\(Q:\mathbb{R}^{m}\setminus\{0\}\rightarrow \s^{n}\)
given by
\begin{align*}
Q:=\big(\sin\gamma\cdot \nu^{(q)},\cos\gamma\big),   
\qquad\gamma\in(0,\pi/2),
\end{align*}
is a proper triharmonic map if and only if the equation
\begin{align}
\label{eq:tri-rotation}
&3\sin^4\gamma\,q ^3(m+q -2)^3-2\sin^2\gamma\,q ^2(m+q -2)^2[4+(m+q -6)(4+q )+q (m+q -2)] \\
\nonumber&+q (q +2)(q +4)(m+q -2)(m+q -4)(m+q -6)=0
\end{align}
is satisfied.
\end{Satz}

A careful analysis shows that \eqref{eq:tri-rotation} can always be solved for \(m>2\), see
\cite[Corollary 4.7]{BS25tri} for the precise details.

A variational problem of arbitrary order for maps to the sphere
is represented by extrinsic polyharmonic maps. To recall their definition
let $u:M\rightarrow \s^{n}$ be a map and $\iota:\s^{n}\rightarrow\mathbb{R}^{n+1}$ be the canonical embedding. 
We will deal with maps in the Sobolev space
\begin{align*}
W^{k,2}(M,\s^{n})=\{ u\in W^{k,2}(M,\mathbb{R}^{n+1})\,\colon\, u(x)=(u_1(x),\dots, u_{n+1}(x))\in\s^{n}\,\mbox{almost everywhere}\},   
\end{align*}
where $k\in\mathbb{N}$ and $(u_1,\dots,u_{n+1})$
denote the component functions of $\iota\circ u$.

The \textit{extrinsic $k$-energy functional} for $u\in W^{k,2}(M,\s^{n})$ is defined by
\begin{align}
\label{poly-energies}
E_k^{ext}(u)&=\frac{1}{2}\int_{M}|\Delta^su|^2\dv,\qquad\mbox{when}\,\, k=2s, s\in\mathbb{N};\\
\nonumber E_k^{ext}(u)&=\frac{1}{2}\int_M|\nabla\Delta^su|^2\dv,\qquad\mbox{when}\,\,  k=2s+1,  s\in\mathbb{N}_0.
\end{align}
The critical points of the extrinsic $k$-energy functional are referred to as \textit{extrinsic $k$-harmonic maps} or \textit{extrinsic polyharmonic maps} of order \(k\) and are characterized as solutions of
\begin{align*}
    \Delta^ku-\langle\Delta^ku,u\rangle u=0.
\end{align*}

An important class of critical points of the extrinsic $k$-energy functional is made up of the so-called \textit{minimizers (of the extrinsic $k$-energy functional)}.
More precisely, a map $u\in W^{k,2}(B^m,\s^{n})$ is called minimizer if $E_k^{ext}(u)\leq E_k^{ext}(v)$ for all $v\in W^{k,2}(B^m,\s^{n})$ with $u-v\in W_0^{k,2}(B^m,\mathbb{R}^{n+1})$.
Furthermore, an extrinsic $k$-harmonic map is called \textit{stable} if the second variation of \eqref{poly-energies}
is positive, i.e.
\begin{align*}
    \frac{d^2}{dt^2}E_k^{ext}(u_t)\lvert_{t=0}\geq 0
\end{align*}
for all variations $u_t$ with $u_t-u\in W_0^{k,2}(B^m,\mathbb{R}^{n+1})$. Otherwise, the map $u$ is referred to as \textit{unstable}. 

For the standard radial projection a stability analysis was carried out by Fardoun, Montaldo and Ratto in \cite{MR4436204}. The stability of the generalized equator map (\ref{equator-generalized})
considered as an extrinsic polyharmonic map
was investigated in
\cite{BS25kstable} by the authors and can be extended to the maps constructed in Theorem\,\ref{thm:main}.
We now recall the main results that were obtained on this topic:

\begin{Thm}
\label{thm:exst}
Assume \(m\geq 2k+1\), where $m,k\in\mathbb{N}$. 
There exists a polynomial \(Q_k^q(m)\) such that the generalized equator map \(\nu_\star^{(q)}\colon B^m\to\s^{n}\) is
\begin{enumerate}
    \item an energy minimizing extrinsic \(k\)-harmonic map if \(Q_k^q(m)\geq 0\),
    \item an unstable extrinsic \(k\)-harmonic map if \(Q_k^q(m)<0\).
\end{enumerate}
\end{Thm}

The previous theorem allows to prove the existence of energy minimizing extrinsic $k$-harmonic maps:

\begin{Thm}
\label{thm:st2}
For any $k\geq 2$ and $q\in\mathbb{N}$ we have $Q_k^q(m)>0$ for any $m\geq 5+3(k-2)+5q$. If this inequality is satisfied, the generalized equator map $\nu_\star^{q}:B^m\rightarrow\s^{n}$ is an energy minimizing extrinsic $k$-harmonic map.
\end{Thm}

Eventually, for \(q=m\) the stability analysis yields the following instability result:

\begin{Thm}
\label{thm:st3}
 Assume that \(q=m\geq 2k+1\). Then, the generalized equator map \(\nu_\star^{(q)}\colon B^m\to\s^{n}\) is unstable. 
\end{Thm}

We would like to emphasize that all theorems mentioned above, specifically Theorems\,\ref{thm:exst}, \ref{thm:st2}, and \ref{thm:st3}, can be applied to any generalized radial projection or the generalized equator map \eqref{equator-generalized_2}, respectively, as derived from Theorem \ref{thm:main}.

\medskip

\textbf{Notation and Conventions:}
Throughout this article we will employ the following sign conventions: 
For the Riemannian curvature tensor field we use 
$$
R(X,Y)Z=[\nabla_X,\nabla_Y]Z-\nabla_{[X,Y]}Z,
$$ 
where \(X,Y,Z\) are vector fields.

For the rough Laplacian on the pull-back bundle $\phi^{\ast} TN$ we employ the analysts sign convention, i.e.
$$
\bar\Delta = \tr(\bar\nabla\bar\nabla-\bar\nabla_\nabla).
$$
In particular, this implies that the Laplace operator has a negative spectrum.

\medskip

\textbf{Organization:}
We provide preliminaries in Section\,\ref{sec:prelim}.
In Section\,\ref{sec-nak} we reveal the underlying theoretical framework
of the general radial projection introduced by Nakauchi in \cite{MR4371934,MR4593065}. Furthermore, we construct additional new generalized equator maps.

\section{Preliminaries}
\label{sec:prelim}
This section provides the mathematical background for our analysis. In Subsection\,\ref{sub-Nakauchi} we first recall the \lq generalized radial projection\rq\, introduced by Nakauchi in \cite{MR4371934,MR4593065}, and the \lq generalized equator map\rq\, introduced in \cite{BS25kstable}. Subsection\,\ref{sub-eigenmaps} contains a brief review on eigenmaps after which we discuss the 
so-called harmonic projection operator in Subsection\,\ref{sub-hmp}. Finally, Subsection \ref{sub-op} presents a brief exposition on operators acting on eigenmaps.

\subsection{Generalized radial projection and generalized equator map}
\label{sub-Nakauchi}
\subsubsection{Generalized radial projection}
We give a short overview of the generalized radial projection due to Nakauchi \cite{MR4371934,MR4593065},
where we closely follow the presentation in \cite{BS25kstable}.

\smallskip

In \cite[Main Theorem, p.1]{MR4593065} Nakauchi showed that for any $\ell,m\in\mathbb{N}$ with $\ell\leq m$ there exists a harmonic map
\begin{align}
\label{nak-maps}
u^{(\ell)}\colon\mathbb{R}^{m}\setminus\{0\}&\to\s^{m^{\ell}-1},\\
\notag x=(x_1,\dots,x_m)&\mapsto u^{(\ell)}(x)=(u^{(\ell)}_{i_1\dots i_{\ell}}(x))_{1\leq i_1,\dots i_{\ell}\leq m},
\end{align}
defined as follows:
We set \(y_i=\frac{x_i}{\norm x}\) and have the recursive definition
\begin{align}
\label{nak-rec}
u^{(1)}_{i_1}(x)=&y_{i_1},\\
\nonumber u^{(\ell)}_{i_1\ldots i_\ell}(x)=&C_{\ell,m}\big(y_{i_1}u^{(\ell-1)}_{i_1\ldots i_{\ell-1}}(x)-\frac{1}{\ell+m-3}\norm x\frac{\partial}{\partial x_{i_\ell}}u^{(\ell-1)}_{i_1\ldots i_{\ell-1}}(x)\big),
\end{align}
where 
\begin{align*}
 C_{\ell,m}=\sqrt{\frac{\ell+m-3}{2\ell+m-4}}.   
\end{align*}
The map $u^{(\ell)}$ is henceforth referred to as \textit{generalized radial projection}.

\subsubsection{Identities for the generalized radial projection}
We will now present some identities of the generalized radial projection. Let $\ell,m\in\mathbb{N}$ such that $\ell\leq m$.
Nakauchi \cite{MR4593065} proved that the map $u^{(\ell)}\colon\mathbb{R}^{m}\setminus\{0\}\to\s^{m^\ell-1}$ satisfies: 
\begin{enumerate}
\item $u^{(\ell)}$ solves the equation for harmonic maps to spheres
\begin{align*}
    \Delta u^{(\ell)}+\norm{\nabla u^{(\ell)}}^2 u^{(\ell)}=0; 
\end{align*}
\item $u^{(\ell)}$ is a polynomial in $u_{i_1},\dots, u_{i_{\ell}}$ of degree $\ell$, where $u_{i_j}=\frac{x_{i_j}}{\norm x}$;
\item $\norm{\nabla u^{(\ell)}}^2=\frac{\ell(\ell+m-2)}{\norm x^2}$.
\end{enumerate}
Furthermore, Nakauchi \cite[Proposition 1, (1)]{MR4593065} showed that the map $u^{(\ell)}$ satisfies the identity
\begin{align}
\label{ortho}
   \sum_{j=1}^m y_j\cdot\nabla_ju^{(\ell)}=0.
\end{align}

The following Lemma, which is Lemma\,2.3 from \cite{BS25tri}, gives an explicit expression for $\Delta ^{k}u^{(\ell)}$ where $k\in\mathbb{N}$.

\begin{Lem}
\label{lem:delta-nak}
 For each $k\in\mathbb{N}$ the map \(u^{(\ell)}\) defined in (\ref{nak-maps}) satisfies 
 \begin{align*}
          \Delta ^{k}u^{(\ell)}=
          \big(\prod_{j=1}^k(2j+\ell-2)(2j-\ell-m)\big)\frac{u^{(\ell)}}{\norm x^{2 k}}.
    \end{align*} 
\end{Lem}

Additional properties of the map \(u^{(\ell)}\) can be found in the article of Nakauchi \cite{MR4593065}. We will not mention them here as they will not be used in this manuscript.

\subsubsection{Generalized equator map}
The \textit{generalized equator map} can be constructed naturally from the generalized radial projection as follows: 
\begin{align} \label{equator-generalized-a} u_\star^{(l)}\colon B^m&\to\s^{m^{\ell}}\subset\mathbb{R}^{m^{\ell}}\times\mathbb{R},\\ \notag x&\mapsto\big(u^{(l)},0\big), \end{align} where \(u^{(l)}\) is specified in \eqref{nak-maps}.
For $\ell=1$ the map \eqref{equator-generalized-a} is the well-known equator map, i.e.
\begin{align}
  u_\star^{(1)}\colon B^m&\to\s^{m}\subset\mathbb{R}^{m}\times\mathbb{R},\\ 
  \notag x&\mapsto\big(\frac{x}{\norm{x}},0\big).
\end{align}    
Note that we consider $u_\star^{(\ell)}$ as a one-parameter family of maps and use the singular to refer to it.
Finally, note that the identities of the generalized radial projection mentioned in the preceding subsection induce identities for the generalized equator map in a natural manner.

\subsection{Eigenmaps}
\label{sub-eigenmaps}
In this subsection we provide some background on eigenmaps.
In Section\,\ref{sec-nak} we then establish a link between eigenmaps and the map introduced by Nakauchi
in \cite{MR4371934,MR4593065}.

\smallskip

We begin by recalling the definition of eigenmaps.

\begin{Dfn}[see e.g. page 77 in \cite{MR2044031}]
\label{def:eigen}
Denote by $\iota:\s^n\rightarrow\R^{n+1}$ the inclusion map.
A smooth map $\varphi:M\rightarrow\s^n$ is called an \textit{eigenmap} if all components of the map $\Phi:=\iota\circ\varphi$ are eigenfunctions of the Laplacian on $M$ with the same eigenvalue.
\end{Dfn}

By the following proposition, a smooth map $\varphi:M\rightarrow\s^n$ 
is an eigenmap if and only if it is harmonic with constant energy density.

\begin{Prop}[Proposition 3.3.17 in \cite{MR2044031}]
 Let the maps $\varphi, \iota$ and $\Phi$ be as in Definition\,\ref{def:eigen}. Then $\varphi$ is harmonic if and only if
 \begin{align*}
     \Delta\Phi=\lambda\Phi,
 \end{align*}
 for some function $\lambda:M\rightarrow\R$. Further, in this case, we have $\lambda=-\lvert d\Phi\rvert^2=-\lvert d\varphi\rvert^2$.
\end{Prop}

Eigenmaps do not only play an important role within the theory of harmonic maps, but also have strong connections to other classical objects from differential geometry, see e.g. \cite{MR3969432,MR5063505} as well as the references therein.
Note that most Riemannian manifolds do not admit eigenmaps, since the spectrum of the Laplacian is generically simple, see \cite{MR464332}.

\smallskip

In this manuscript we focus on the case $M$ being a round sphere.
A map $\varphi:\s^m\rightarrow\s^r$ is an eigenmap of degree $p$, or short $\lambda_p$-eigenmap, if all components of $\varphi$ are spherical harmonics of order $p$ on $\s^m$. 
Recall that a spherical harmonic on $\s^m$ of order $p$ is the restriction to $\s^m$ of a homogeneous harmonic polynomial of degree $p$ in $m+1$ variables. 
Equivalently, a spherical harmonic is
an eigenfunction of $\Delta^{\s^m}$ with eigenvalue $\lambda_p=p(p+m-1)$. 
The space of homogeneous harmonic polynomials of degree $p$ in $m+1$ variables will be denoted by $\mathcal{H}^{\R^{m+1}}_{p}$.
Further, the space of spherical harmonics on $\s^m$ of order $p$ will be denoted by $\mathcal{H}^{\s^m}_{p}$.

\smallskip

Below we provide some examples for eigenmaps between spheres.

\begin{enumerate}
    \item \textbf{Standard minimal immersion:}\\
It is well known, see e.g. \cite[Chapter VIII, Section 1]{MR1242555} that
\begin{align*}
\dim\mathcal{H}^{\s^m}_{p}=n(\lambda_p)+1=\frac{(2p+m-1)(p+m-1)!}{(p+1)!(m-1)!}.
\end{align*}
Let
$\{f^j_{\lambda_p}\}_{j=0}^{n(\lambda_p)}\subset\mathcal{H}^{\s^m}_{p}$ be
an orthonormal basis of $\mathcal{H}_{p}(m+1)$ with respect to the scalar product
\begin{align*}
 \langle h_1, h_2\rangle =\frac{n(\lambda_p)+1}{\vol(\s^m)}\int_{\s^m}h_1h_2\dv.
\end{align*}
Then, the map given by
\begin{align*}
f_{\lambda_p}\colon\s^m\rightarrow\s^{n(\lambda_p)}\subset \mathcal{H}^{\s^m}_{p},\,x\mapsto \sum_{j=0}^{n(\lambda_p)}f^j_{\lambda_p}(x)f^j_{\lambda_p}
\end{align*}
is a $\lambda_p$-eigenmap called \textit{standard minimal immersion (of degree $p$ in $m+1$ variables)}, see e.g. \cite[page 66]{MR1280796} or \cite{MR1042100}.
    \item \textbf{Veronese map:}\\
    The Veronese map $V:\s^2\rightarrow\s^4$ is an eigenmap between spheres of degree $2$ which is given by the restriction to $\s^2$ of the map $\Phi:\mathbb{R}^3\rightarrow\mathbb{R}^5$ determined by 
\begin{align*}
(x_1,x_2,x_3)\mapsto \big(\frac{\sqrt{3}}{2}(x_1^2-x_2^2), \frac{1}{2}(x_1^2+x_2^2-2x_3^2), \sqrt{3}x_1x_2, \sqrt{3}x_2x_3, \sqrt{3}x_1x_3\big). 
\end{align*}   
    \item \textbf{Eiconals:}\\ 
    An eiconal is a homogeneous polynomial $E:\R^m\rightarrow\R$ of degree $d$, $d\in\N_0$, if it satisfies
    \begin{align*}
       \| (\mbox{grad}\,E)_x\|^2=\| x\|^{2d} 
    \end{align*}
    for $x\in\R^m$. For such $E$, the map $F:\R^m\rightarrow\R^m$ given by $\varphi=\mbox{grad}\,E_{\lvert \s^{m-1}}:\s^{m-1}\rightarrow\s^{m-1}$ is an eigenmap of degree $d$.
    This construction is due to R. Wood (unpublished note, compare \cite[pages 80 and 103]{MR2044031}).
    \item \textbf{Orthogonal multiplications:}\\ A bilinear map $b:\R^m\times\R^n\rightarrow\R^r$ which satisfies
    \begin{align*}
        \|b(x,y)\|= \| x\| \| y\|
    \end{align*}
    for all $x\in\R^m, y\in\R^n$, is called \textit{orthogonal multiplication}. For $m=n$, the Hopf construction applied to the orthogonal multiplication $b$ is given by
    \begin{align*}
        H(b)(x,y)=(\| x\|^2-\| y\|^2,2b(x,y));
    \end{align*} 
    it gives rise to an eigenmap $\s^{2m-1}\rightarrow\s^r$ of degree $2$, see e.g. \cite[page 62]{MR1280796} or  \cite[pages 205-206]{MR1232826}.\\ Examples of orthogonal multiplications are given by the standard multiplications of the real, complex, quaternionic and Cayley numbers. These multiplications induce the so-called \textbf{Hopf fibrations} $\s^{3}\rightarrow\s^2$, $\s^{7}\rightarrow\s^4$ and $\s^{15}\rightarrow\s^8$, respectively.
\end{enumerate}

An eigenmap $f:\s^m\rightarrow\s^V$ is called \textit{full} if $\mbox{span}\,\mbox{Im}(f)=V$, i.e. the image of $f$ spans $V$.
Any eigenmap $f$ gives rise to a full eigenmap, which we also denote by $f$.

\begin{Bem}
The eiconals also appear in the construction of harmonic self-maps of spheres, compare the first Proposition in \cite[page 304]{MR3193778} and \cite[Theorem\,C]{MR4000241}. 
\end{Bem}

\subsection{Harmonic projection operator}
\label{sub-hmp}
In this subsection we give a survey on the basic ideas of the so-called \lq Harmonic projection operator\rq\, of homogeneous polynomials introduced in \cite{MR229863}.

\smallskip

Set $\mathcal{R}_{\ell}(n)$ to be the space  of homogeneous polynomials of degree $\ell$ in $n$ variables and $\mathcal{H}_{\ell}^{\R^n}$ be the space of homogeneous harmonic polynomials of degree $\ell$ in $n$ variables.
Elementary considerations show that $\mathcal{H}_{\ell}(n)$ is a subspace
of $\mathcal{R}_{\ell}(n)$.

In \cite{Fischer1918} Fischer proved that (see e.g. the introduction of \cite{MR4743396} for a modern presentation of this result), for a given homogeneous polynomial $p(y)$, $y\in\R^n$,
every homogeneous polynomial $f$ of degree $\ell$ can be uniquely decomposed as follows:
\begin{align*}
      f(y)=h_{\ell}(y)+p(y)f_1(y).  
\end{align*}
Here, $f_1$ is a homogeneous polynomial and
$h_{\ell}$ is a homogeneous polynomial of degree $\ell$ which satisfies the partial differential equation
$p(D)h_{\ell}=0$, where $D$ denotes the differential operator which is associated to $y$ by the Fourier identification, i.e. $y_j$ corresponds to $\frac{\partial}{\partial x_j}$ for $j\in\{1,\dots,n\}$.

For the specific case $p(y)=\sum_{i=1}^ny_i^2$ we have $p(D)=\sum_{i=1}^n\frac{\partial^2}{\partial x_i^2}$ and hence $h_{\ell}$ is harmonic. 

Therefore, the following result is a specific case of the Fischer decomposition:

\begin{Prop}[compare page 444 in \cite{MR229863}]\label{prop-proj}
 The space $\mathcal{R}_{\ell}(n)$ of homogeneous polynomials of degree $\ell$ in $n$ variables is the direct sum of the subspace $\mathcal{H}_{\ell}^{\R^n}$ of homogeneous harmonic polynomials of degree $\ell$ in $n$ variables and the subspace $r^2\mathcal{R}_{\ell-2}(n)$ of polynomials of the form $r^2f_1$, where $f_1\in\mathcal{R}_{\ell-2}(n)$. 
\end{Prop}

Let $f\in\mathcal{R}_{\ell}(n)$. By Proposition\,\ref{prop-proj} we have
\begin{align*}
    f(x)=h_{\ell}(x)+r^2f_1(x),
\end{align*}
where $f_1\in\mathcal{R}_{\ell-2}(n)$ and $h_{\ell}\in\mathcal{H}_{\ell}^{\R^n}$.
The polynomial $h_{\ell}$ is called the \textit{harmonic projection of the homogeneous polynomial $f$} and it is henceforth denoted by $H(f)$. $H$ is referred to as \textit{harmonic projection operator}.

\smallskip

In \cite[page 446]{MR229863}, an explicit formula for the harmonic projection of $f$ has been obtained:
\begin{align}
\label{hmp-general}
H(f)(x)=\sum_{k=0}^{\lfloor\frac{\ell}{2}\rfloor}\frac{(-1)^kr^{2k}\Delta^kf(x)}{2^kk!(n+2\ell-4)\dots(n+2\ell-2k-2)}.    
\end{align}
Now, assume that, in addition, $f$ is harmonic, i.e. $f\in\mathcal{H}_{\ell}^{\R^n}$.
We obtain the identity
\begin{align}
\label{hmp}
H(x_if)(x)=x_if-\frac{r^2}{m+2p-2}\frac{\partial f}{\partial x_i},
\qquad 1\leq i\leq m.
\end{align}
Indeed, this identity is an immediate consequence of (\ref{hmp-general}),
since the degree of $x_if$ is $\ell+1$ and furthermore
\begin{align*}
 \Delta(x_if(x))=2\frac{\partial f}{\partial x_i}   
\end{align*}
which implies 
\begin{align*}
\Delta^k(x_if(x))=0
\end{align*}
for any $k\geq 2$. 

\subsection{Operators on eigenmaps}
\label{sub-op}
In this subsection we provide some basics on 
operators on eigenmaps between spheres which have been introduced by
 Toth in \cite{MR1241953}. Throughout let $m\in\N$.

 \smallskip

  An operator on eigenmaps between spheres associates a $\lambda_q$-eigenmap with a $\lambda_p$-eigenmap in a natural way.
The next theorem shows that homomorphisms of
orthogonal $\mbox{SO}(m+1)$-modules into the tensor product $(\mathcal{H}_p^{\s^m})^{\star}\otimes\mathcal{H}^{\s^m}_q$ represent a large class of such operators. In order to state the theorem we first need
to define the mathematical setup properly.

In the following, $D$ will be a nonzero module homomorphism of an orthogonal $\mbox{SO}(m+1)$-module $W$ into $(\mathcal{H}_p^{\s^m})^{\star}\otimes\mathcal{H}_q^{\s^m}$. For $e\in W$, the map $D_e:\mathcal{H}_p^{\s^m}\rightarrow \mathcal{H}_q^{\s^m}$ is linear.
Dually, we can think of \(D\) as an $\mbox{SO}(m+1)$-module homomorphism 
\(\iota\colon\mathcal{H}_p^{\s^m}\to W\otimes\mathcal{H}_p^{\s^m}\).
Then, \(D\) and \(\iota\) are related via 
\begin{align*}
\iota^T(e\otimes h)=D_eh,\qquad
\iota(h')=\sum_{i=0}^me_i\otimes D^{T}_{e_i}h' \text{ for } h'\in\mathcal{H}_q^{\s^m},   
\end{align*}
where $\{e_i\}_{i=0}^m\subset W$ is an orthonormal basis. We have
\begin{align*}
    \iota^{T}\circ\iota=c(D)\cdot Id,
\end{align*}
with $c(D)\in\mathbb{R}$, see \cite[equation $(8)$]{MR1241953}.
Let $f:\mathbb{R}^{m+1}\rightarrow\mathbb{R}^{n+1}$ be a polynomial map whose components are in $\mathcal{H}_p^{\R^{m+1}}$. Set $f^D:\mathbb{R}^{m+1}\rightarrow W\otimes\mathbb{R}^{n+1}$ to be the map consisting of the components 
\begin{align*}
    (f^D)_{i}^j=c(D)^{-\frac{1}{2}}D_{e_i}f^j.
\end{align*}
With this preparation at hand we can now state the already above mentioned theorem:

\begin{Thm}[Theorem 1 in \cite{MR1241953}]\label{thm:operators}
    Let $D$ be a nonzero module homomorphism of an orthogonal $\mbox{SO}(m+1)$-module $W$ into $(\mathcal{H}_p^{\s^m})^{\star}\otimes\mathcal{H}_q^{\s^m}$. Given a $\lambda_p$-eigenmap $f:\s^m\rightarrow\s^n$, $f^D$ maps the unit sphere into the unit sphere of $W\otimes\R^{n+1}$ which we denote by $S_{W\otimes\R^{n+1}}$, i.e. $f^D:\s^m\rightarrow S_{W\otimes\R^{n+1}}$ is a $\lambda_q$-eigenmap.
\end{Thm}

Next, we list several examples of operators on eigenmaps, compare  \cite{MR1241953}.

\begin{Bsp}
\begin{enumerate}
    \item Define $D^{\pm}$ as $\mbox{SO}(m+1)$-module homomorphism of
$\mathcal{H}_1^{\s^m}$ into $(\mathcal{H}_p^{\s^m})^{\star}\otimes\mathcal{H}_{p\pm 1}^{\s^m}$ by
\begin{align*}
 D_e^+h=\mu_p\sum_{i=0}^mc_iH(x_ih),\quad  D_e^-h=\mu_{p-1}^{-1}\sum_{i=0}^mc_i\frac{\partial h}{\partial x_i},    
\end{align*}
where $H$ is the harmonic projection operator, $h\in\mathcal{H}_p^{\R^{m+1}}$, $$e=\sum_{i=0}^mc_ix_i\in \mathcal{H}_1^{\s^m}$$
and $$\mu_p=(p+1)\frac{2p+m-1}{p+m-1}.$$ Note that $D^+$ is a degree raising operator (the degree is raised by $1$) and $D^-$ is a degree lowering operator (the degree is lowered by $1$).

Moreover, note that in Euclidean space \(\R^p\) with coordinates \(x_i,i=1,\ldots,p\) dilatations are generated by the operator
\(\sum_{i=0}^px_i\frac{\partial}{\partial x_i}\).
Hence, the existence of the operator \(D_e^-h\) arises due to one of the symmetries of Euclidean space. 
\item 
Consider the operator acting on $\mathcal{H}_p^{\s^m}$ which is given by
\begin{align*}
    D_e^{1,0}=(p-1)!\sum_{i,k=0}^mc_{i,k}A_{i,k},
\end{align*}
where
$$A_{i,k}=x_k\frac{\partial}{\partial x_i}-x_i\frac{\partial}{\partial x_k}.$$
Given a $\lambda_p$-eigenmap $\varphi:\s^m\rightarrow\s^n$, then $\varphi^D:\s^m\rightarrow S_V$ is again a $\lambda_p$-eigenmap. Here, we have $V=\R^{n+1}\otimes\mbox{SO}(m+1)$ and $S_V$ denotes the Euclidean sphere of radius $1$ in $V$.
Note that the coordinates of $\varphi^D$ are obtained from those of $\varphi$ by infinitesimal rotations in each coordinate plane of $\R^{m+1}$ and that \(A_{i,k}\) is the generator of rotations in Euclidean space.
\item 
Another operator acting on $h\in \mathcal{H}_p^{\s^m}$ is given by
\begin{align*}
    D_e^{2,0}h=(p-1)!\sum_{i,k=0}^mc_{i,k}H(x_k\frac{\partial h}{\partial x_i}+x_i\frac{\partial h}{\partial x_k}),
\end{align*}
where $H$ is the harmonic projection operator and $c_{i,k}\in\R$.
The operator \(x_k\partial x_i+x_i\partial x_k\) is not connected to any symmetry of Euclidean space but rather represents a deformation of it since it does not preserve the Euclidean distance.
\end{enumerate}
\end{Bsp}

\section{Generalized radial projection maps}
\label{sec-nak}
In this section, we reveal the theoretical background used to construct the generalized radial projection \eqref{equator-generalized}. 
These considerations are crucial for the remainder of this manuscript.

\smallskip

By $u^{(n)}$ we denote, as above, the generalized radial projection.
Given this data we now manufacture a new map $v^{(n)}\colon\mathbb{R}^{m}\setminus\{0\}\to\mathbb{R}^{m^{\ell}}$ by setting
\begin{align}
\label{def_v}
v^{(n)}(x):=\norm{x}^n\,u^{(n)}(x).
\end{align}
The next lemma shows that \(v^{(n)}(x)\) is a harmonic homogeneous polynomial.

\begin{Lem}
\label{lem_har}
 The map $v^{(n)}\colon\mathbb{R}^{m}\setminus\{0\}\to\mathbb{R}^{m^{\ell}}$ is a harmonic homogeneous polynomial of degree $n$. In particular, $v^{(n)}$ can be extended uniquely and smoothly to a harmonic homogeneous polynomial of degree $n$ on  $\mathbb{R}^{m}$.   
\end{Lem}
\begin{proof}
An induction argument, which makes use of the recursive definition of $u^{(n)}$, shows that $v^{(n)}\colon\mathbb{R}^{m}\setminus\{0\}\to\mathbb{R}^{m^{\ell}}$ is a homogeneous polynomial in the variables $x_1,\dots,x_m$ of degree $n$. Consequently, it can be extended uniquely and smoothly to a harmonic homogeneous polynomial of degree $n$ on  $\mathbb{R}^{m}$.   

\smallskip

The product rule for the Laplacian yields
\begin{align*}
 \Delta u^{(n)}(x)=\norm{x}^{-n}\Delta v^{(n)}(x)+2\sum_{i=1}^n\frac{\partial}{\partial x_i}(\norm{x}^{-n})\frac{\partial}{\partial x_i}(v^{(n)}(x))+\Delta(\norm{x}^{-n})v^{(n)}(x).    
\end{align*}
Since $v^{(n)}(x)$ is a homogeneous polynomial of degree $n$,
we can apply the Euler identity, see e.g.\cite[p. 287]{MR248290}, which gives
\begin{align*}
\sum_{i=1}^nx_i\frac{\partial}{\partial x_i}v^{(n)}(x)=n\,v^{(n)}(x).    
\end{align*}
A straightforward computation yields
\begin{align*}
\frac{\partial}{\partial x_i}\norm{x}^{-n}=-n\norm{x}^{-(n+2)}x_i.    
\end{align*}
The two preceding identities thus imply
\begin{align*}
2\sum_{i=1}^n\frac{\partial}{\partial x_i}(\norm{x}^{-n})\frac{\partial}{\partial x_i}(v^{(n)}(x))=-2n^2\norm{x}^{-(n+2)}\,v^{(n)}(x).    
\end{align*}
Further, a direct calculation implies
\begin{align}
\label{eq:delta}
 \Delta(\norm{x}^{-n})=n(n+2-m)\norm{x}^{-(n+2)}.   
\end{align}
Therefore, we obtain
\begin{align*}
 \Delta u^{(n)}(x)=\norm{x}^{-n}\Delta v^{(n)}(x)-n(n+m-2)\norm{x}^{-(n+2)}v^{(n)}(x).    
\end{align*}
Using Lemma \ref{lem:delta-nak} we find
\begin{align*}
 \Delta u^{(n)}(x)=-n(n+m-2)\norm{x}^{-2}u^{(n)}(x),        
\end{align*}
such that
\begin{align*}
 \Delta v^{(n)}(x)=0,       
\end{align*}
which establishes the claim.
\end{proof}

In \cite{MR1241953} Toth introduced a method to generate new eigenmaps between spheres from given eigenmaps
which we now briefly recall. To this end, let $f:\s^{m-1}\rightarrow\s^{d}$ be a given $\lambda_p$-eigenmap, where $\lambda_p=p(p+m-2)$\footnote{Note that we use the convention $f:\s^{m-1}\rightarrow\s^{d}$, while Toth considers maps $f:\s^{m}\rightarrow\s^{d}$.}.
Then, the polynomial map 
\begin{align*}
f^{+}\colon&\mathbb{R}^m\rightarrow\mathbb{R}^m\otimes\mathbb{R}^{d+1},\\
(f^+)^j_i&=c_p^{+}H(x_if^{j}),
\end{align*}
where $1\leq i\leq m$ and $1\leq j\leq d+1$,
when restricted to the Euclidean sphere, yields an eigenmap between spheres.
Here, we have
\begin{align*}
c_p^{+}&=\sqrt{\frac{2p+m-2}{p+m-2}}\in\mathbb{R},\\
H(x_if^{j})&=x_if^j-\frac{r^2}{2p+m-2}\frac{\partial f^j}{\partial x_i},
\end{align*}
where $H$ is the \textit{harmonic projection operator}, 
which was first introduced in \cite{MR229863}, see also Subsection\,\ref{sub-hmp} for more details.
We will now choose $f$ to be $v^{(n-1)}$, which was defined in (\ref{def_v}), whose restriction to $\s^{m-1}$ is an eigenmap by Lemma\,\ref{lem_har}.
Let $n\in\N$ with $n\geq 2$ be given.
A straightforward calculation yields
\begin{align*}
H\big(x_iv_j^{(n-1)}(x)\big)=&x_iv_j^{(n-1)}(x)-\frac{1}{m+2n-4}\norm{x}^2\frac{\partial}{\partial x_i}v_j^{(n-1)}(x)\\=& \frac{m+n-3}{m+2n-4}x_i \norm{x}^{n-1}u_j^{(n-1)}(x)-\frac{1}{m+2n-4}\norm{x}^{n+1}\frac{\partial}{\partial x_i}u_j^{(n-1)}(x)\\=& \frac{m+n-3}{m+2n-4}\norm{x}^{n}\big(x_i \norm{x}^{-1}u_j^{(n-1)}(x)-\frac{1}{m+n-3}\norm{x}\frac{\partial}{\partial x_i}u_j^{(n-1)}(x)\big). 
\end{align*}
Thus, we get
\begin{align*}
 c_{n-1}^{+}H(x_iv_j^{(n-1)}(x))=\sqrt{\frac{m+n-3}{m+2n-4}}\norm{x}^{n}\big(x_i\norm{x}^{-1}u_j^{(n-1)}(x)-\frac{1}{m+n-3}\norm{x}\frac{\partial}{\partial x_i}u_j^{(n-1)}(x)\big). 
\end{align*}

\smallskip

This calculation yields a link to \cite[Definition 1, Formula (3.2)]{MR4593065}, which is the recursive definition \eqref{nak-rec}.
More precisely, this recursive definition can be rewritten as
\begin{align*}
u_{i,j}^{(n)}=\frac{c_{n-1}^{+}}{\norm{x}^{n}}H\big(x_i\norm{x}^{n-1}u_j^{(n-1)}(x)\big).
\end{align*}
The above formula provides a conceptual explanation for the recursive definition \eqref{nak-rec} of the generalized radial projection. 

The previous observation motivates the following definition.
\begin{Dfn}
\label{def-op}
    Let $\varphi:\s^{m-1}\rightarrow\s^d$ be an eigenmap and denote the associated $(d+1)$-tuple by $\Phi:\mathbb{R}^{m}\rightarrow\mathbb{R}^{d+1}$.
Let $\ell$ be the degree of the homogeneous polynomials which are the components of $\Phi$.
Set 
\begin{align*}
 \nu^{(\ell)}(x):=\frac{\Phi(x)}{\norm{x}^{\ell}},
\end{align*}
and 
\begin{align*}
\nu^{(n)}\colon\mathbb{\R}^{m}\setminus\{0\}\rightarrow&\mathbb{R}^{m^{n-\ell}(d+1)},\\
x \mapsto&\frac{c_{n-1}^{+}}{\norm{x}^{n}}H\big(x_i\norm{x}^{n-1}\nu^{(n-1)}(x)\big),
\end{align*}    
for $n\in\mathbb{N}$ with $n\geq\ell+1$. 

In other words, we define the operator 
$R:(\mathcal{H}^{\R^m}_{\ell})^{d+1}\times \mathbb{\R}^{m}\setminus\{0\}\rightarrow\mathbb{R}^{m(d+1)}$
by
\begin{align*}
R(\Phi,x)\colon=(\frac{c_{n-1}^{+}}{\norm{x}^{n}}H(x_i\Phi))_{1\leq i\leq m}.
\end{align*}
\end{Dfn}

\begin{Bem}
\begin{enumerate}
\item Note that the operator $R$ equals exactly the operator $D^+$, which we discussed in Subsection\,\ref{sub-op}.
    \item 
The considerations before Definition\,\ref{def-op} show that
    in \cite{MR4593065}  Nakauchi considered the case $\ell=1$, $d=m-1$ and $\Phi(x)=x$.
    \end{enumerate}
\end{Bem}

Our next goal is to show that the maps $\nu^{(n)}$ given by Definition\,\ref{def-op} satisfy the properties (1)-(3) discussed in Subsection\,\ref{sub-Nakauchi}, i.e. 
\begin{enumerate}
\item $\nu^{(n)}$ satisfies the equation for harmonic maps to spheres
\begin{align*}
    \Delta \nu^{(n)}+\norm{\nabla\nu^{(n)}}^2 \nu^{(n)}=0; 
\end{align*}
\item $\nu^{(n)}$ is a polynomial in $u_{i_1},\dots, u_{i_{n}}$ of degree $n$, where 
$u_{i_j}=\frac{x_{i_j}}{\norm{x}}$;
\item $\norm{\nabla \nu^{(n)}}^2=\frac{n(n+m-2)}{\norm x^2}$.
\end{enumerate}
For this purpose, we first establish two preparatory lemmas.

\begin{Lem}
\label{lem:delta}
    Let $f:\mathbb{R}^m\rightarrow\mathbb{R}$ be a homogeneous harmonic polynomial of degree $d$ in $m$ variables. Then, the formula
    \begin{align*}
        \Delta(\frac{f}{\norm{x}^d})=-\frac{d(d+m-2)}{\norm{x}^2}\,\frac{f}{\norm{x}^d}
    \end{align*}
    holds true.
\end{Lem}
\begin{proof}
 This follows from the product rule for the Laplacian, the Euler formula and identity (\ref{eq:delta}).   
\end{proof}

\begin{Lem}[See VIII, Corollary 1.10 in \cite{MR1242555}]
The eigenmaps $\varphi:\s^{m-1}\rightarrow\s^d$ are represented by $(d+1)$-tuples of harmonic $\ell$-homogeneous polynomials
\begin{align*}
    (\Phi^{\alpha})_{1\leq\alpha\leq d+1}\,\,\mbox{in}\,\,\mathcal{H}_{\ell}^{\R^m}\,\,\mbox{for some}\,\,\ell\in\mathbb{N}
\end{align*}
with
\begin{align*}
    \sum_{\alpha=1}^{d+1}\Phi^{\alpha}(x)^2=1\,\,\mbox{for all}\,\,x\in\s^{m-1}.
\end{align*}
\end{Lem}

As announced above, the following theorem shows that the maps $u^{(n)}$ introduced in Definition \,\ref{def-op} satisfy the properties (1)-(3) discussed in Subsection \,\ref{sub-Nakauchi} and thus represent a significant generalization of Nakauchi's construction. 

\begin{Thm}
\label{thm:main-sec3}
Let $\varphi:\s^{m-1}\rightarrow\s^d$ be an eigenmap and denote the associated $(d+1)$-tuple by $\Phi:\mathbb{R}^{m}\rightarrow\mathbb{R}^{d+1}$.
Let $\ell$ be the degree of the homogeneous polynomials which are the components of $\Phi$.
Set 
\begin{align*}
 \nu^{(\ell)}(x):=\frac{\Phi(x)}{\norm{x}^{\ell}},
\end{align*}
and 
\begin{align}
\label{def_nu}
\nu^{(n)}\colon\mathbb{\R}^{m}\setminus\{0\}&\rightarrow\mathbb{R}^{m^{n-\ell}(d+1)},\\
\nonumber x& \mapsto\frac{c_{n-1}^{+}}{\norm{x}^{n}}H(x_i\norm{x}^{n-1}\nu^{(n-1)}(x)),
\end{align}    
for $n\in\mathbb{N}$ with $n\geq\ell+1$. 
The maps $\nu^{(n)}$ are homogeneous polynomials in the variables $(y_i)_{i=1,\dots,m}$, where $y_i\colon=\frac{x_i}{\norm{x}}$
and satisfy the following properties:
\begin{enumerate}
    \item $\nu^{(n)}$ is sphere-valued, i.e. $\norm{\nu^{(n)}(x)}^2=1$ for all $x\in\mathbb{\R}^{m}\setminus\{0\}$;
    \item the energy density of $\nu^{(n)}$ is given by
    \begin{align*}
 \norm{\nabla  \nu^{(n)}(x)}^2=\frac{n(n+m-2)}{\norm{x}^2};
\end{align*}
 \item $\nu^{(n)}$ is a harmonic map, i.e. it is a solution of the differential equation
 \begin{align*}
     \Delta \nu^{(n)}+\norm{\nabla  \nu^{(n)}}^2\nu^{(n)}=0.
 \end{align*}
\end{enumerate}
\end{Thm}

\begin{proof}
We prove the claim by induction and start with the base case.

\smallskip

For any $x\in\mathbb{R}^{m}\setminus{\{0\}}$ we have
$\norm{\Phi(\frac{x}{\norm{x}})}^2=1$ and hence $\norm{\Phi(x)}^2=\norm{x}^{2\ell}$. Hence $ \nu^{(\ell)}$ is sphere-valued.

\smallskip

By the very definition of $ \nu^{(\ell)}$ it follows that each component of $\nu^{(\ell)}(x)$ is a polynomial in $\frac{x_1}{\norm{x}},\dots,\frac{x_m}{\norm{x}}$ of degree $\ell$.

\smallskip

In addition, we have
\begin{align*}
 \norm{\nabla  \nu^{(\ell)}(x)}^2=\frac{\ell(\ell+m-2)}{\norm{x}^2}.   
\end{align*}
Indeed: Recall that for an arbitrary homogeneous polynomial $F$ of degree $\ell$ we have
\begin{align*}
\nabla F  =\ell F \frac{x}{\norm{x}}+\norm{x}^{\ell-1}\nabla^{\s^{m-1}}F_{\lvert\s^{m-1}}.  
\end{align*}
Hence, we get
\begin{align*}
\nabla \nu^{(\ell)}(x)&=\frac{\nabla\Phi}{\norm{x}^{\ell}}+\nabla({\norm{x}^{-\ell}})\Phi\\&= \frac{1}{\norm{x}^{\ell}}\big(\ell\Phi \frac{x}{\norm{x}^2}+\norm{x}^{\ell-1}\nabla^{\s^{m-1}}\Phi_{\lvert\s^{m-1}}\big)
- \ell\Phi \frac{x}{\norm{x}^{\ell+2}}= \norm{x}^{-1}\nabla^{\s^{m-1}}\Phi_{\lvert\s^{m-1}}. 
\end{align*}
Thus, the claim follows from the identity $\norm{\nabla^{\s^{m-1}}\Phi_{\lvert\s^{m-1}}}^2=\ell(\ell+m-2)$ which holds since $\Phi_{\lvert\s^{m-1}}$ is an eigenmap.

\smallskip

From Lemma\,\ref{lem:delta} we thus get
that 
\begin{align*}
    \Delta \nu^{(\ell)}+\norm{\nabla \nu^{(\ell)}}^2\nu^{(\ell)}=0,
\end{align*}
i.e. $\nu^{(\ell)}$ is a harmonic map.

\smallskip

The induction step follows along the lines of the proof of \cite[Proposition 2]{MR4593065}.
Nevertheless, we provide the proof for the sake of completeness.
Assume that for one fixed $n\in\N$ with $n\geq \ell+1$, the map $\nu^{(n-1)}$ satisfies 
\begin{align}
\label{ass-1}
\norm{\nu^{(n-1)}(x)}^2=1
\end{align}
for all $x\in\mathbb{\R}^{m}\setminus\{0\}$, as well as the identities
    \begin{align}
    \label{ass-2}
 \norm{\nabla  \nu^{(n-1)}(x)}^2=\frac{(n-1)(n+m-3)}{\norm{x}^2}
\end{align}
and
 \begin{align}
    \label{ass-3}
     \Delta \nu^{(n-1)}+\norm{\nabla  \nu^{(n-1)}}^2\nu^{(n-1)}=0.
 \end{align}

\smallskip

\textit{Claim $1$}: Identities (\ref{ass-1}) and (\ref{ass-2}) yield that $\norm{\nu^{(n)}(x)}^2=1$.
Indeed, by taking the first derivative of (\ref{ass-1}) we obtain
\begin{align*}
 \langle \nabla  \nu^{(n-1)}(x),  \nu^{(n-1)}(x)\rangle=0,   
\end{align*}
where $\langle\,\cdot\,,\,\cdot\,\rangle$ denotes the standard scalar product.
Hence, from (\ref{def_nu}) and using the assumptions (\ref{ass-1}) and (\ref{ass-2}) we get 
\begin{align*}
\norm{\nu^{(n)}(x)}^2=1,
\end{align*}
which establishes Claim $1$.

\smallskip

\textit{Claim $2$}: Identities (\ref{ass-2}), (\ref{ass-3}) and $\norm{\nu^{(n)}(x)}^2=1$ imply the equations 
    \begin{align}
    \label{id-1}
 \norm{\nabla  \nu^{(n)}(x)}^2=\frac{n(n+m-2)}{\norm{x}^2}
\end{align}
and
 \begin{align}
 \label{id-2}
     \Delta \nu^{(n)}+\norm{\nabla  \nu^{(n)}}^2\nu^{(n)}=0.
 \end{align}
Indeed, applying the Laplacian to (\ref{def_nu}) and making use of (\ref{ass-2}) as well as (\ref{ass-3}), a straightforward computation yields that the identity (\ref{id-2}) holds true.
From $\norm{\nu^{(n)}(x)}^2=1$ we obtain by differentiating that 
\begin{align*}
 \langle \nabla  \nu^{(n)}(x),  \nu^{(n)}(x)\rangle=0   
\end{align*}
holds. 
Differentiating again yields 
\begin{align*}
   \norm{ \nabla \nu^{(n)}(x)}^2+\langle  \nu^{(n)}(x),\Delta  \nu^{(n)}(x)\rangle=0. 
\end{align*}
The identity (\ref{id-1}) thus follows by making use of 
(\ref{id-2}) and $\norm{\nu^{(n)}(x)}^2=1$.

\smallskip

Clearly, combining Claim $1$ and Claim $2$ establishes the induction step, whence the proof.
\end{proof}
    
Next we present a number of examples of $\nu^{(n)}$ for some $n\in\N$ and a few prominent choices of $\Phi$.

\begin{Bsp}
\begin{enumerate}
    \item 
The Veronese map $V:\s^2\rightarrow\s^4$ is an eigenmap between spheres which is given by the restriction to $\s^2$ of the map $\Phi:\mathbb{R}^3\rightarrow\mathbb{R}^5$ determined by 
\begin{align*}
(x_1,x_2,x_3)\mapsto \bigg(\frac{\sqrt{3}}{2}(x_1^2-x_2^2), \frac{1}{2}(x_1^2+x_2^2-2x_3^2), \sqrt{3}x_1x_2, \sqrt{3}x_2x_3, \sqrt{3}x_1x_3\bigg). 
\end{align*}    
Each component function of $\Phi$ is a homogeneous harmonic polynomial of degree $\ell=2$.  
We have
\begin{align*}
 \nu^{(2)}(x):=\frac{\Phi(x)}{\norm{x}^{2}}.   
\end{align*}
Below we determine  $\nu^{(3)}:\mathbb{R}^3\setminus\{0\}\rightarrow\s^{14}$. By definition we have 
\begin{align*}
\nu_{i,j}^{(3)}(x) =\frac{c_{2}^{+}}{\norm{x}^{3}}H(x_i\norm{x}^{2}\nu_j^{(2)}(x)),
\end{align*} 
with $c_{2}^{+}=\sqrt\frac{5}{3}$, $i\in\{1,2,3\}$ and $j\in\{1,2,3,4,5\}$.
Straightforward computations give
\begin{align*}
 (\nu_{1,j}^{(3)}(x))_{1\leq j\leq 5}=\frac{1}{\norm{x}^3}\bigg(&\frac{x_1(3x_1^2-7x_2^2-2x_3^2)}{2\sqrt5}, 
 \sqrt{\frac{3}{5}}\frac{x_1(x_1^2+x_2^2-4x_3^2)}{2}, \\
 &\frac{x_2(4x_1^2-x_2^2-x_3^2)}{\sqrt{5}}, 
 \sqrt{5}x_1x_2x_3, \frac{x_3(4x_1^2-x_2^2-x_3^2)}{\sqrt{5}}\bigg),\\   
 (\nu_{2,j}^{(3)}(x))_{1\leq j\leq 5}=\frac{1}{\norm{x}^3}\bigg(&\frac{x_2(7x_1^2-3x_2^2+2x_3^2)}{2\sqrt5}, 
 \sqrt{\frac{3}{5}}\frac{x_2(x_1^2+x_2^2-4x_3^2)}{2},\\ 
 &\frac{x_1(-x_1^2+4x_2^2-x_3^2)}{\sqrt{5}}, \frac{x_3(-x_1^2+4x_2^2-x_3^2)}{\sqrt{5}}, 
 \sqrt{5}x_1x_2x_3\bigg), \\  
 (\nu_{3,j}^{(3)}(x))_{1\leq j\leq 5}=\frac{1}{\norm{x}^3}\bigg(&\frac{\sqrt{5}(x_1^2-x_2^2)x_3}{2}, 
 \sqrt{\frac{3}{5}}\frac{x_3(3x_1^2+3x_2^2-2x_3^2)}{2},\\&\sqrt{5}x_1x_2x_3, 
 \frac{x_2(-x_1^2+x_2^2-4x_3^2)}{\sqrt{5}}, \frac{x_1(-x_1^2-x_2^2+4x_3^2)}{\sqrt{5}}\bigg).   
\end{align*}
Note that $\nu^{(3)}$ is not a full eigenmap. 
By a straightforward computation we find that $\mbox{rank}\,(\nu^{(3)}) =7$. 
Hence, there exists an associated full eigenmap
$\tilde v^{(3)}:\s^2\rightarrow\s^6$.
Since $\dim \mathcal{H}_{3}(3)=7$ the map $\tilde v^{(3)}$ is a standard minimal immersion. Thus, by \cite[Lemma $1$]{MR1241953}, 
each $v^{(n)}:=\norm{x}^n\nu^{(n)}$ with $n\geq 3$ is equivalent to the corresponding standard minimal immersion.
\item The Hopf fibration $H_2:\s^3\rightarrow\s^2$ is an eigenmap between spheres which is given by the restriction to $\s^3$ of the map $\Phi:\mathbb{R}^4\rightarrow\mathbb{R}^3$ determined by 
\begin{align*}
(x_1,x_2,x_3,x_4)\mapsto \big(x_1^2+x_2^2-x_3^2-x_4^2, 2(x_1x_3+x_2x_4), 2(x_1x_4-x_2x_3)\big). 
\end{align*}    
Each component function of $\Phi$ is a homogeneous harmonic polynomial of degree $\ell=2$.  
We have
\begin{align*}
 \nu^{(2)}(x):=\frac{\Phi(x)}{\norm{x}^{2}}.   
\end{align*}
Below we determine  $\nu^{(3)}:\mathbb{R}^4\setminus\{0\}\rightarrow\s^{11}$. By the very definition of $\nu^{(n)}$ we have
\begin{align*}
\nu_{i,j}^{(3)}(x) =\frac{c_{2}^{+}}{\norm{x}^{3}}H(x_i\norm{x}^{2}\nu_j^{(2)}(x)),
\end{align*} 
with $c_{2}^{+}=\sqrt\frac{3}{2}$, $i\in\{1,2,3,4\}$ and $j\in\{1,2,3\}$.
Straightforward computations give
\begin{align*}
 (\nu_{1,j}^{(3)}(x))_{1\leq j\leq 3}&=\frac{1}{\norm{x}^3}\bigg(\sqrt{\frac{2}{3}}x_1(x_1^2+x_2^2-2(x_3^2+x_4^2)),\\&\frac{5x_1^2x_3+6x_1x_2x_4-x_3(x_2^2+x_3^2+x_4^2)}{\sqrt6}, \\&\frac{-6x_1x_2x_3+5x_1^2x_4-x_4(x_2^2+x_3^2+x_4^2)}{\sqrt6}\bigg),\\
(\nu_{2,j}^{(3)}(x))_{1\leq j\leq 3}&=\frac{1}{\norm{x}^3}\bigg(\sqrt{\frac{2}{3}}x_2(x_1^2+x_2^2-2(x_3^2+x_4^2)),\\&\frac{-6x_1x_2x_3+x_1^2x_4+x_4(-5x_2^2+x_3^2+x_4^2)}{\sqrt6},\\&\frac{6x_1x_2x_4+5x_1^2x_3+x_3(-5x_2^2+x_3^2+x_4^2)}{\sqrt6}\bigg),\\
(\nu_{3,j}^{(3)}(x))_{1\leq j\leq 3}&=\frac{1}{\norm{x}^3}\bigg(-\sqrt{\frac{2}{3}}x_3(-2x_1^2-2x_2^2+(x_3^2+x_4^2)),\\&\frac{x_1^3+6x_2x_3x_4-x_1(x_2^2-5x_3^2+x_4^2)}{\sqrt6},\\&\frac{x_1^2x_2+6x_1x_3x_4+x_2(x_2^2-5x_3^2+x_4^2)}{\sqrt6}\bigg),\\
(\nu_{4,j}^{(3)}(x))_{1\leq j\leq 3}&=\frac{1}{\norm{x}^3}\bigg(-\sqrt{\frac{2}{3}}x_4(-2x_1^2-2x_2^2+(x_3^2+x_4^2)),\\&\frac{-x_1^2x_2+6x_1x_3x_4-x_2(x_2^2+x_3^2-5x_4^2)}{\sqrt6},
\\&-\frac{x_1^3+6x_2x_3x_4+x_1(x_2^2+x_3^2-5x_4^2)}{\sqrt6}\bigg). 
\end{align*}
Further, $\dim\,\mbox{im}(\nu^{(3)})=8$.
Since $\dim \mathcal{H}_{3}(4)=16$ the map $\nu^{(3)}$ is not equivalent to the standard minimal immersion.
\item 
The map $H_1:\s^1\rightarrow\s^1$ is an eigenmap between spheres which is given by the restriction to $\s^1$ of the map $\Phi:\mathbb{R}^2\rightarrow\mathbb{R}^2$ determined by 
\begin{align*}
(x_1,x_2)\mapsto \big(x_1^2-x_2^2, 2x_1x_2\big). 
\end{align*}    
Each component function of $\Phi$ is a homogeneous harmonic polynomial of degree $\ell=2$.  
We have
\begin{align*}
 \nu^{(2)}(x):=\frac{\Phi(x)}{\norm{x}^{2}}.   
\end{align*}
Note that $v^{(n)}(x)$ induces a full eigenmap $\tilde{v}^{(n)}(x):\s^1\rightarrow\s^1$ which gives rise to  $\tilde{\nu}^{(n)}(x):=\frac{\tilde{v}^{(n)}(x)}{\norm{x}^{n}}$ with $\tilde{\nu}^{(n)}(x):\mathbb{R}^2\rightarrow\s^1$.
Then, the components of $\tilde{v}^{(n)}$ are given by the (rescaled) real bivariate harmonic polynomials of degree $n$.
\item In \cite{MR4593065} Nakauchi considered the case
\begin{align*}
 \nu^{(1)}(x)=\frac{x}{\norm{x}}.   
\end{align*}
Since $v^{(1)}=\norm{x}\nu^{(1)}$ is a standard minimal immersion,
\cite[Lemma $1$]{MR1241953} yields that 
each $v^{(n)}:=\norm{x}^n\nu^{(n)}$ with $n\geq 1$ is equivalent to the corresponding standard minimal immersion.
\end{enumerate}
\end{Bsp}

Using Theorem\,\ref{thm:operators} we can now easily generalize 
Theorem\,\ref{thm:main-sec3}:

\begin{Thm}
\label{thm:main2}
Let $\varphi:\s^{m-1}\rightarrow\s^{d-1}$ be an eigenmap and denote the associated $d$-tuple by $\Phi:\mathbb{R}^{m}\rightarrow\mathbb{R}^{d}$.
Let $p$ be the degree of the homogeneous polynomials which are the components of $\Phi$.
Set 
\begin{align*}
 \nu^{(p)}(x):=\frac{\Phi(x)}{\norm{x}^{p}}.
\end{align*}
Let $D$ be a nonzero module homomorphism of an orthogonal $\mbox{SO}(m)$-module $W$ into $(\mathcal{H}_p)^{\star}\otimes\mathcal{H}_q$.
We define $\nu^{(q)}$ as follows
\begin{align*}
\nu^{(q)}\colon&\mathbb{\R}^{m}\setminus\{0\}\rightarrow\mathbb{R}^{d \dim W},\\
&x \mapsto\frac{1}{\norm{x}^{q}}(\norm{x}^{p}\nu^{(p)}(x))^D.
\end{align*}    
The map $\nu^{(q)}$ is a homogeneous polynomial in the variables $(y_i)_{i=1,\dots,m}$, where $y_i\colon=\frac{x_i}{\norm{x}}$
and satisfies the following properties:
\begin{enumerate}
    \item $\nu^{(q)}$ is sphere-valued, i.e.$\norm{\nu^{(q)}(x)}^2=1$ for all $x\in\mathbb{\R}^{m}\setminus\{0\}$;
    \item the energy density of $\nu^{(q)}$ is given by
    \begin{align*}
 \norm{\nabla  \nu^{(q)}(x)}^2=\frac{q(q+m-2)}{\norm{x}^2};
\end{align*}
 \item $\nu^{(q)}$ is a harmonic map, i.e. it is a solution of the differential equation
 \begin{align*}
     \Delta \nu^{(q)}+\norm{\nabla  \nu^{(q)}}^2\nu^{(q)}=0.
 \end{align*}
\end{enumerate}
\end{Thm}

We omit the proof of Theorem\,\ref{thm:main2} since it is analogous to that of Theorem\,\ref{thm:main-sec3}.

\smallskip

Clearly, one can apply the same construction method to $\nu^{(q)}$.
Thus one recursively defines a sequence $(\nu^{(n_i)})_{i\in\mathbb{N}}$ of maps satisfying properties (1)-(3) of the statement of Theorem\,\ref{thm:main2}. Note that the sequence $n_i$ is determined by $D$, $d$, $n_1$ and $\dim W$.

\bibliographystyle{plain}
\bibliography{mybib}

\end{document}